%% file: trans-sm.tex
\documentclass[twoside, final]{amsart} 

\usepackage{amsmath, amsfonts, amsthm, amssymb}
\usepackage{exscale}
\usepackage{cite}
\usepackage{epsfig}
\usepackage{amscd}
\usepackage{graphics}
\usepackage{color}
\usepackage{dsfont}

\usepackage{showkeys}

\renewcommand{\geq}{\geqslant}
\renewcommand{\leq}{\leqslant}

\newtheorem{theorem}{Theorem}

\newtheorem{proposition}{Proposition}[section]

\newtheorem{lemma}[proposition]{Lemma}
\newtheorem*{main-theorem}{Main Theorem}
\newtheorem*{theorem*}{Theorem}

\theoremstyle{definition}

\newtheorem{remark}[proposition]{Remark}

\newtheorem*{remark*}{Remark}

\numberwithin{equation}{section}

\def\phi{\varphi}

\def\ZZ{{\mathbb Z}}
\def\NN{{\mathbb N}}
\def\reals{{\mathbb R}}

\def\Ci{{\mathcal C}^\infty}

\def\Im{\,\mathrm{Im}\,}

\def\sgn{\mathrm{sgn}\,}

\def\supp{\mathrm{supp}\,}

\def\O{{\mathcal O}}

\def\s{{\mathcal S}}

\def\Op{\mathrm{Op}\,}

\def\csh{{\left( h/\tilde{h} \right)}}
\def\phi{\varphi}

\def\be{\begin{eqnarray*}}
\def\ee{\end{eqnarray*}}
\def\ben{\begin{eqnarray}}
\def\een{\end{eqnarray}}

\def\lll{\left\langle}
\def\rrr{\right\rangle}

\def\L2R{L_{\text{Rest}}^2}

\def\11{\mathds{1}}

\def\L2c{L^2_{\text{comp}}}

\def\th{{\tilde{h}}}

\def\tDelta{\widetilde{\Delta}}

\def\tP{\widetilde{P}}

\def\tu{\tilde{u}}

\def\C{\mathcal{C}}

\def\Vol{\text{Vol}}

\def\tg{\tilde{g}}

\def\p{\partial}

\def\p{\partial}

\newcommand{\abs}[1]{{\left\lvert{#1}\right\rvert}}
\newcommand{\norm}[1]{{\left\lVert{#1}\right\rVert}}
\newcommand{\ang}[1]{{\left\langle{#1}\right\rangle}}
\newcommand{\pa}{{\partial}}
\newcommand{\ep}{{\epsilon}}
\newcommand{\hamvf}{{\textsf{H}}}

\newcommand{\B}{\mathcal{B}}

\newcommand{\la}{{\langle}}
\newcommand{\ra}{{\rangle}}
\newcommand{\cd}{{\,\cdot\,}}
\newcommand{\R}{{\mathbb{R}}}

\begin{document}

\title[Local smoothing with transmission]{Sharp local smoothing for
  manifolds with smooth inflection transmission}

\author{Hans Christianson}
\email{hans@math.unc.edu}
\address{Department of Mathematics, UNC-Chapel Hill \\ CB\#3250
  Phillips Hall \\ Chapel Hill, NC 27599}

\author{Jason Metcalfe}
\email{metcalfe@email.unc.edu}
\address{Department of Mathematics, UNC-Chapel Hill \\ CB\#3250
  Phillips Hall \\ Chapel Hill, NC 27599}

\subjclass[2000]{}
\keywords{}

\begin{abstract}

We consider a family of spherically symmetric, asymptotically
Euclidean manifolds with two trapped sets, one which is unstable and
one which is semi-stable.  The phase space structure is that of an
inflection transmission set.  We prove a sharp local smoothing estimate for
the linear Schr\"odinger equation with a loss which depends on how
flat the manifold is near each of the trapped sets.  The result
interpolates between the family of similar estimates in
\cite{ChWu-lsm}.  
As a consequence of the techniques of proof, we also show a sharp high
energy 
resolvent estimate with a polynomial loss depending on how flat the
manifold is near each of the trapped sets.

\end{abstract}

\maketitle

\section{Introduction}
\label{S:intro}

In this paper we study the local smoothing effect for the Schr\"odinger
equation on a class of manifolds with a trapped set that is mixed
unstable and {\it semistable}, which is a version of inflection
transmission.  
Our main result is a generalization of the local smoothing estimate
\[\int_0^T \|\la x\ra^{-1/2-}
e^{it\Delta}u_0\|^2_{H^{1/2}}\,dt\lesssim \|u_0\|^2_{L^2}.\]
Such estimates first appeared in \cite{ConSau}, \cite{Sl}, \cite{Vega} and were extended to
nontrapping asymptotically flat geometries in \cite{CKS},
\cite{MR1373768}.  See, e.g., \cite{MR2333213}, \cite{MR2565717} for some recent generalizations.
The presence of trapping
necessitates a loss of smoothing as was shown in \cite{Doi}.  
If the trapping is unstable and
nondegenerate, this has already been studied in
\cite{Bur-sm},\cite{Chr-NC,Chr-disp-1,Chr-QMNC},\cite{Dat-sm}, \cite{BGH} amongst several
others.  Trapping that is unstable but degenerately so was the topic of
\cite{ChWu-lsm}.  The novel thing in this
paper is the existence of semistable trapping, that is, trapping which
is stable from one direction and unstable from another direction.

Let us begin by describing the geometry.  Let $m_1$ and $m_2$ be
positive integers, and set 
\[
a(x) = x^{2m_1-1}(x-1)^{2m_2}/(1 + x^2)^{m_1 +m_2-1},
\]
so that
\[
a(x) \sim \begin{cases} x^{2m_1-1}, \quad x \sim 0, \\
(x-1)^{2m_2}/2^{m_1+m_2-1} \quad x \sim 1, \\
x, \quad |x| \to \infty. \end{cases}
\]
Set 
\[
A^2(x) = 1 + \int_0^x a(y) dy,
\]
and notice that
\begin{equation}\label{A2}
A^2(x) \sim \begin{cases} 1 + x^{2m_1}, \quad x \sim 0, \\
C_1 + c_2(x-1)^{2m_2 +1} \quad x \sim 1, \\
x^2, \quad | x | \to \infty. \end{cases}
\end{equation}
Here $C_1 >1$ and $c_2<1$ are constants which are easily computed but
inessential, except for their relative sizes compared to $1$.

Now let $X = \reals_x \times \reals_\theta / 2 \pi
\ZZ$, equipped with the metric 
\[
ds^2 = d x^2 + A^2(x) d \theta^2,
\]
so that $X$ is asymptotically Euclidean with two ends and has two
trapped sets.  The trapping occurs where $A'(x)=0$, which is at $x = 0$ and $x
= 1$ respectively (see Figure \ref{fig:mfld}).

 \begin{figure}
  \hfill
  \centerline{\input{mfld}}
  \caption{\label{fig:mfld} A piece of the manifold $X$ with the trapped
    sets at $x = 0$ and at $x= 1$.   }
  \hfill
  \end{figure}
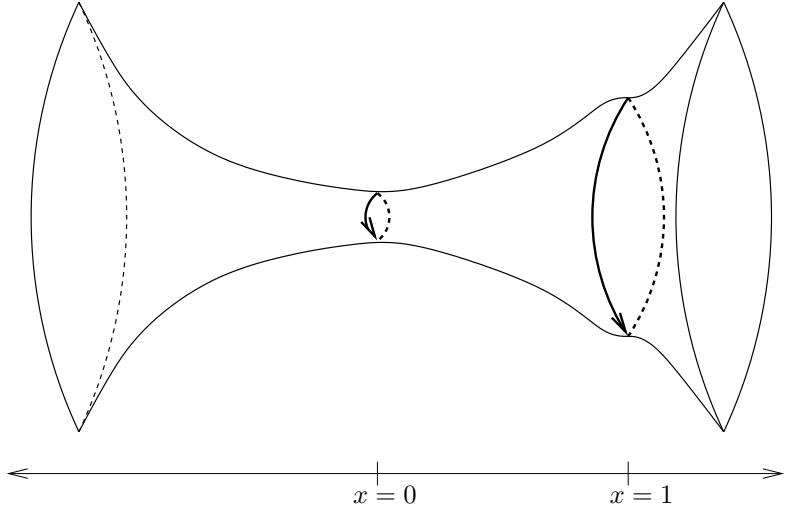


The metric determines the volume form
\[
d \Vol = A(x) dx d \theta
\]
and the Laplace-Beltrami operator acting on $0$-forms
\[
\Delta f = (\partial_x^2 + A^{-2} \partial_\theta^2 + A^{-1}
A' \partial_x) f.
\]
We conjugate $\Delta$ and reduce to a one dimensional problem.
Indeed, we set $L : L^2(X, d\Vol) \to L^2(X, dx d \theta)$ to be the
isometry
\[Lu(x,\theta)=A^{1/2}(x)u(x,\theta).\]
With mild assumptions on $A$, $\tDelta = L\Delta L^{-1}$ is
(essentially) 
self-adjoint on $L^2(X,dx d\theta)$.  More explicitly, we have
\[-\tDelta f = (-\partial_x^2 -A^{-2}(x)\partial_\theta^2 + V_1(x))f\]
with
\[V_1(x)=\frac{1}{2}A''A^{-1} - \frac{1}{4}(A')^2 A^{-2}.\]

Given a function $\psi$ on $X$, we expand into its Fourier series,
$\psi(x,\theta)=\sum_k \phi_k(x)e^{ik\theta}$, and note that
\[(-\tDelta - \lambda^2)\psi = \sum_k e^{ik\theta} (P_k-\lambda^2)\phi_k(x),\]
where 
\[P_k\phi_k(x) = \Bigl(-\frac{d^2}{dx^2} + k^2 A^{-2}(x) + V_1(x)\Bigr)\phi_k(x).\]
By setting $h=k^{-1}$, we pass to the semiclassical operator
\[(P(h)-z)\phi(x) = \Bigl(-h^2\frac{d^2}{dx^2} + V(x)-z\Bigr)\phi(x),\]
where the potential is
\[V(x)=A^{-2}(x)+h^2 V_1(x)\]
and the spectral parameter is $z=h^2\lambda^2$.

Our main result is the following local smoothing estimate with sharp
loss.  Using the common notation $D_t = (1/i)\partial_t$, we have:

\begin{theorem}[Local Smoothing]
\label{T:smoothing}

Suppose $X$ is as above with $m_1, m_2 \geq 1$ and assume $u$ solves
\[
\begin{cases} (D_t -\Delta ) u = 0 \text{ in } \reals \times X , \\
u|_{t=0} = u_0 \in H^{s}
\end{cases}
\]
for some $s >0$ sufficiently large.  
Then for any $T<\infty$, there exists a constant
$C_T >0$ such that 
\begin{align*}
\int_0^T & \Bigl( \| \lll x \rrr^{-1} \partial_x u \|_{L^2(d\Vol)}^2 + \| \lll x
\rrr^{-3/2} \partial_\theta u \|_{L^2(d\Vol)}^2 \Bigr) \, dt \\
&\qquad\qquad \leq C_T \Bigl(\| \lll D_\theta
\rrr^{\beta(m_1, m_2) } u_0 \|_{L^2(d\Vol)}^2 + \| \lll D_x \rrr^{1/2} u_0 \|_{L^2(d\Vol)}^2\Bigr),
\end{align*}
where
\begin{equation}\label{beta}
\beta(m_1, m_2) = 
\max \left( \frac{m_1}{m_1 + 1}, \frac{2m_2 + 1}{2m_2 +3} 
\right). 
\end{equation}

Moreover this estimate is sharp, in the sense that no polynomial
improvement in regularity is true.

\end{theorem}

This theorem requires some remarks.

\begin{remark}
Observe that the maximum {\it
  gain} in regularity in the presence of inflection-transmission
trapping is $2/(2m_2 + 3)$ derivatives.  Each of these fractions lies
in between sequential fractions in the numerology of \cite{ChWu-lsm},
since
\[
\frac{1}{(m+1)+1} < \frac{2}{2m+3} < \frac{1}{m+1}.
\]
\end{remark}

\begin{remark}

In the theorem above, the weights at infinity are different than those
that appear in the standard Euclidean estimate.  Standard cutoff
arguments would allow us to make these match.  The key new aspect of
the theorem, however, is the behavior near the trapped sets, and for clarity in
the proof, we do not modify the weights.  

\end{remark}

\begin{remark}\label{separates}

Theorem \ref{T:smoothing} and indeed also Theorem \ref{T:resolvent}
below are of course true in many more situations.  Of particular
interest, the microlocalization step which separates the trapped sets
at different energies used to prove \eqref{glued}
indicates the same result applies to a manifold with one Euclidean end
and only an inflection transmission trapped set.  

On the other hand, if our manifold has two Euclidean ends, a
degenerate hyperbolic trapped set, and {\it two} inflection
transmission trapped sets {\it at the same} semiclassical energy, it
is natural to suspect that such a theorem is no longer true because
the two inflection transmission sets must tunnel to each other.
However, it is easy to see that the theorem still applies in this
case, since the stable/unstable manifolds for the degenerate
hyperbolic trapped set form a separatrix (in other words, the
degenerate hyperbolic trapped set is at {\it higher} semiclassical
energy).  Hence the same microlocalization applies, and so does the theorem.

\end{remark}

\begin{remark}
We briefly discuss why we have chosen to call this
type of smooth trapping ``inflection-transmission'' type trapping.
The inflection part refers to the fact that the effective potential
after separating variables has an inflection point at the trapped
set.  We have also included transmission in our name because this kind
of trapping bears some resemblance to the traditional transmission
problem.  

The traditional transmission problem concerns a wave equation in a
medium for which the speed of propagation is distinct in different
regions.  For example, one might study solutions to the equation
\[
\begin{cases}
(\p_t^2 - \Delta)u = 0 \text{ for } | x | <1, \\
(\p_t^2 - c^2 \Delta) u = 0 \text{ for } | x | >1,
\end{cases}
\]
where $c \neq 1$.  Of course one also needs to indicate appropriate
boundary conditions at the interaction surface where $| x | = 1$ (see,
for example, \cite{CPV-trans-I,CPV-trans-II,CaVo-trans}).

On the other hand, if we consider a surface of revolution given by a
specific 
generating curve $C$ in the $(x_1 , x_3)$ plane, rotated around the
$x_3$ axis, we get a similar looking picture.  Let 
\[
C = \{ (x_1, 0, x_3) = (A(r), 0, B(r)), r \geq 0 \},
\]
where 
\[
A(r) = \begin{cases} r, \text{ for } 0 \leq r \leq 1, \\\frac{1}{2}r, \text{ for } r \geq 3,
\end{cases}
\]
and assume $0 \leq A'(r) \leq 1$ and $A$ has an inflection point at,
say, $r = 2$.  The function $A(r)$ is sketched schematically in Figure
\ref{fig:AB}.
 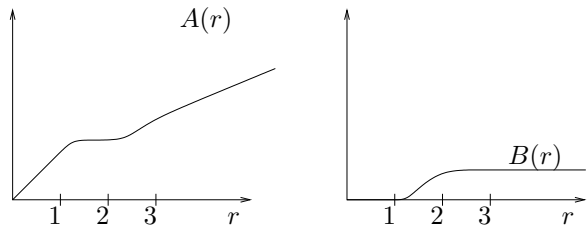
\begin{figure}
 \hfill
 \centerline{\input{AB}}
 \caption{\label{fig:AB}  The functions $A(x)$ and $B(x)$.}
 \hfill
 \end{figure}

We suppose that $B'$ is compactly supported in the region $1\le r\le
3$ and fix $B(0)=0$.  The function $B(r)$ is also depicted in Figure \ref{fig:AB}.


Rotating the curve $C$ about the $x_3$ axis in $\reals^3$ yields a
manifold which is flat near $0$ and flat outside a compact set, and
changes ``height'' in between (see Figure \ref{fig:mfld-trans}).

 \begin{figure}
 \hfill
 \centerline{\input{mfld-trans}}
 \caption{\label{fig:mfld-trans}  The manifold obtained by rotating the
 curve $C$ about the $x_3$ axis.}
 \hfill
 \end{figure}
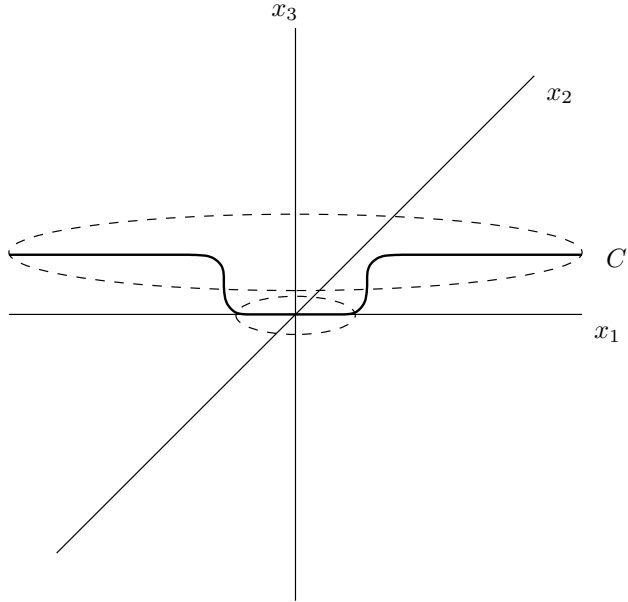

Moreover, if we compute the Laplacian on this surface of revolution,
we see that
\[\Delta_g = \frac{1}{r}\partial_r r\partial_r +
\frac{1}{r^2}\partial_\theta^2,\quad 0\le r\le 1\]
but
\[\Delta_g = 4\Bigl(\frac{1}{r}\partial_r r\partial_r +
\frac{1}{r^2}\partial_\theta^2\Bigr),\quad r\ge 3.\]
See, e.g., \cite{Booth} where such a computation has been carried out
in much detail.

\end{remark}

As in \cite{ChWu-lsm}, once we prove Theorem \ref{T:smoothing}, we can
obtain a resolvent bound.  For simplicity, say that our surface of
revolution is Euclidean at infinity.  That is, assume $A(x)=x$ for
$|x|\gg 0$.  Alternatively, we could require dilation analyticity at
infinity, which would permit asymptotically conic spaces as were
treated in \cite{WZ}.

We let
\[R(\lambda)=(-\Delta_g - \lambda^2)^{-1}\]
denote the resolvent on $X$ (where it exists), and take $\Im
\lambda<0$ as our physical sheet.  With a choice of
appropriate branch cut, $\chi R(\lambda)\chi$ extends meromorphically
to $\{\lambda\in\R\,:\, \lambda\gg 0\}$ for any $\chi\in
C^\infty_c(X)$.  See, e.g., \cite{SjZw}.  And, in the degenerate
inflection point setting, we have

\begin{theorem}
\label{T:resolvent}

For any $\chi \in \Ci_c(X)$, there exists a constant
$C= C_{m_1,m_2, \chi} >0$ such that
for $\lambda \gg 0$, 
\[
\| \chi R(\lambda-i0) \chi \|_{L^2 \to L^2} \leq C \max \{
\lambda^{-2/(m_1+1)}, \lambda^{-4/(2m_2+3)} \}.
\]
Moreover, this is the nut estimate, in the sense that no better
polynomial rate of decay holds true.

\end{theorem}

\subsection*{Acknowledgements} The first author was supported, in
part, by NSF grant DMS-0900524, and the second author by NSF grant
DMS-1054289.

\section{Local smoothing estimates}

In this section, we shall prove the main local smoothing estimate.  In
the subsequent section, we shall saturate the inequality, thus showing
that the loss is sharp.  The proof of the estimate uses a positive commutator
argument.  On Euclidean space, such a proof of local smoothing is
well-known, though the interested reader can see \cite[Section
2.1]{ChWu-lsm} for an exposition which is quite akin to what follows.

We first show
\begin{proposition}
\label{P:smoothing}

Suppose $u$ solves 
\begin{equation}
\label{E:tDelta-Sch}
\begin{cases} (D_t -\tDelta ) u = 0, \\
u(0,x, \theta) = u_0.
\end{cases}
\end{equation}
Then for any $T<\infty$, there exists a constant
$C_T >0$ such that 
\begin{multline*}
\int_0^T  \Bigl( \| \lll x \rrr^{-1} \partial_x u \|_{L^2(dx\,d\theta)}^2 + \| \lll x
\rrr^{-3/2} \partial_\theta u \|_{L^2(dx\,d\theta)}^2 \Bigr) \, dt \\
 \leq C_T \Bigl(\| \lll D_\theta
\rrr^{\beta(m_1, m_2) } u_0 \|_{L^2(dx\,d\theta)}^2 + \| \lll D_x \rrr^{1/2} u_0 \|_{L^2(dx\,d\theta)}^2\Bigr),
\end{multline*}
where $\beta(m_1,m_2)$ is as in \eqref{beta}.
\end{proposition}

The equation \eqref{E:tDelta-Sch} is obtained by conjugating the
original equation by the operator $L$.  Upon conjugating back, Proposition
\ref{P:smoothing} shows that the estimate of Theorem \ref{T:smoothing}
holds.


\subsection{Proof of Proposition \ref{P:smoothing}}
The proof will be broken into three steps.  The first is to use a
positive commutator argument to prove full smoothing away from the
periodic orbits at $x=0$ and $x=1$.  We then expand into a Fourier
series to reduce to a one dimensional problem, and we reduce the
problem to understanding the high frequency part.  Using a $TT^*$ argument, gluing
techniques and a semiclassical rescaling, we show that the high frequency
estimate follows from a cutoff resolvent estimates near each instance of trapping
and subsequently prove those.

\subsubsection{The estimate away from $x=0$ and $x=1$}
For a self-adjoint operator $\tilde{\Delta}$ and a time-independent, self-adjoint multiplier $B$, we
have
\begin{equation}\label{posComm}\frac{d}{dt}\la u, Bu\ra = -2\Im\la (D_t-\tilde{\Delta})u, Bu\ra +i \la
[\tilde{-\Delta}, B]u,u\ra.\end{equation}
In particular, if
\[B=\frac{1}{2}\arctan(x)D_x + \frac{1}{2}D_x \arctan(x),\]
then
\[i[-\tDelta,B] =  2 D_x \lll x \rrr^{-2} D_x + 2 D_\theta A' A^{-3}
\arctan(x) D_\theta  -\frac{3x^2-1}{\la
  x\ra^6}  - V_1'
\arctan(x).\]
Upon integrating \eqref{posComm} over $[0,T]$, we obtain
\[\int_0^T \la i[-\tDelta,B]u,u\ra\,dt = i\la u, \arctan(x) \partial_x
u\ra\Bigl|_0^T - \frac{i}{2}\la u, \la x\ra^{-1} u\ra\Bigl|_0^T\]
for a solution $u$ to \eqref{E:tDelta-Sch}.
Using energy estimates, the right side is controlled by
$\|u_0\|^2_{H^{1/2}}$.  Noting also that energy estimates permit the
control
\[\Bigl|\int_0^T \Bigl\langle \frac{3x^2-1}{\la x\ra^{6}} u +
V_1'\arctan(x)u,u\Bigr\rangle\,dt\Bigr|\leq C T \sup_{0\le t\le T}
\|u(t,\cd)\|^2_{L^2}\leq C_T \|u_0\|^2_{H^{1/2}},\]
it now follows from integration by parts that we have established
\[\int_0^T \Bigl(\|\la x\ra^{-1}\partial_x u\|^2_{L^2} + \|\sqrt{A'
  A^{-3} \arctan(x)}\partial_\theta u\|^2_{L^2}\Bigr)\,dt \leq C_T \|u_0\|^2_{H^{1/2}}.\]

We observe that 
\[
A' A^{-3} \arctan(x) \geq 0
\]
and satisfies
\[
A' A^{-3} \arctan(x) \sim \begin{cases} x^{2m_1}, \quad x \sim 0, \\
c_2' (x-1)^{2m_2}, \quad x \sim 1, \\
|x|^{-3}, \quad | x | \to \infty. \end{cases}
\]
Thus,
\[
\| |x|^{m_1} |x-1|^{m_2}\lll x \rrr^{-m_1 -m_2-3/2} \partial_\theta u
\|_{L^2} \leq C \|
\sqrt{A' A^{-3} \arctan(x) } \partial_\theta u \|_{L^2},
\]
and hence we have the estimate
\begin{equation}
\label{E:est-away-0}
\int_0^T \Bigl(\| \lll x \rrr ^{-1} \partial_x u \|_{L^2}^2 + \|
|x|^{m_1} |x-1|^{m_2}\lll x \rrr^{-m_1 -m_2-3/2} \partial_\theta u
\|_{L^2}^2\Bigr ) \, dt \leq C_T \| u_0\|_{H^{1/2}}^2.
\end{equation}
That is, we have perfect smoothing in the radial direction and in the
$\theta$ direction away from $x = 0$ and $x = 1$, which is precisely
where the trapped sets reside.


\subsubsection{Fourier decomposition}
To get an estimate in the directions
tangential to the trapping, we decompose into Fourier series
\[u(t,x,\theta) = \sum_k e^{ik\theta} u_k(t,x)\]
and
\[u_0(x,\theta) = \sum_k e^{ik\theta} u_{0,k}(x).\]
By Plancherel's theorem, it suffices to show
\begin{multline*}
\int_0^T  \Bigl( \| \lll x \rrr^{-1} \partial_x u_k \|_{L^2(dx)}^2
+ k^2 \| \lll x
\rrr^{-3/2}  u_k \|_{L^2(dx)}^2 \Bigr) \, dt \\
 \leq C_T \Bigl(\| \lll k
\rrr^{\beta(m_1, m_2) } u_{0,k} \|_{L^2(dx)}^2 + \| \lll D_x \rrr^{1/2} u_{0,k} \|_{L^2(dx)}^2\Bigr).
\end{multline*}

We note that as $\partial_\theta u_0=0$, where in an abuse of notation
$u_0$ here stands for the zero mode, the estimate when $k=0$ follows
trivially from \eqref{E:est-away-0}.  Thus, it remains to show
\[
\int_0^T \| \chi(x) k u_k \|_{L^2(\reals)}^2 \, dt \leq C_T\Bigl(  \| \lll k \rrr^{\beta(m_1,m_2)}
u_{0,k} \|_{L^2}^2 + \| u_{0,k} \|_{H^{1/2}}^2\Bigr), \quad |k|\geq 1
\]
for some $\chi \in \Ci_c( \reals)$ with $\chi (x) \equiv 1$ in a
neighborhood of $x=0$ and also in a neighborhood of $x=1$.

In the sequel, we shall be working with a fixed $k$ and as such shall
drop the subscript notation.  Set 
\[P_k = D_x^2 + A^{-2}(x)k^2 + V_1(x).\]
Notice that $P_k$ is merely $-\tDelta$ applied to the $k$th mode.
We fix an even function $\psi\in \C_c^\infty(\R)$ which is $1$ for
$|r|\le \epsilon$ and vanishes for $|r|\ge 2\epsilon$ where
$\epsilon>0$ will be determined later.  Then let
\[u = u_{hi} + u_{lo},\quad u_{hi}=\psi(D_x/k) u.\]

\subsubsection{Low frequency estimate}
We first examine $u_{lo}$ and reduce estimating it to understanding a
bound for $u_{hi}$.
We observe that $u_{lo} = (1-\psi(D_x/k))u$ solves
\[(D_t + P_k) u_{lo} = -[P_k, \psi(D_x/k)] u= k\la x\ra^{-1}
L_k \la x\ra^{-2} \tilde{\psi}(D_x/k)u\]
where $L_k$ is $L^2$ bounded uniformly in $k$ and $\tilde{\psi}\in
\C^\infty_c$ which is identity on the support of $\psi$.

Choosing the same multiplier $B$, replacing $-\tDelta$ with the
self-adjoint $P_k$, and integrating \eqref{posComm} yields
\begin{multline}\label{lowComm}
  \Bigl|\int_0^T \la [P_k,B]u_{lo},u_{lo}\ra\,dt\Bigr| \leq C\Bigl( |\la
  u_{lo}, \arctan(x)\partial_x u_{lo}\ra|\Bigl|_0^T + |\la u_{lo},\la
  x\ra^{-1} u_{lo}\ra|\Bigl|_0^T
\\+ \Bigl|\int_0^T \la \la x\ra^{-1} k L_k \la x\ra^{-2} \tilde{\psi}(D_x/k)u,Bu_{lo}\ra\,dt\Bigr|\Bigr).
\end{multline}
Continuing to argue as above shows that
\[
  \int_0^T \|\la x\ra^{-1}\partial_x u_{lo}\|_{L^2}^2\,dt \leq C_T
  \Bigl(\|u_0\|^2_{H^{1/2}}
 + \Bigl|\int_0^T \la \la x\ra^{-1} k L_k \la x\ra^{-2} \tilde{\psi}(D_x/k)u,Bu_{lo}\ra\,dt\Bigr|\Bigr).
\]
Applying the Schwarz inequality to the last term and bootstrapping, we obtain
\begin{equation}\label{lowComm2}
 \int_0^T \|\la x\ra^{-1}\partial_x u_{lo}\|_{L^2}^2\,dt \leq C_T
  \Bigl(\|u_0\|^2_{H^{1/2}}+
\int_0^T \| k \la x\ra^{-2} \tilde{\psi}(D_x/k)
u\|^2_{L^2}\,dt\Bigr).
\end{equation}
The frequency cutoff guarantees that
\[\int_0^T \|\la x\ra^{-1} k u_{lo}\|^2_{L^2}\,dt \leq C \int_0^T
\|\la x\ra^{-1} \partial_x u_{lo}\|^2_{L^2}\,dt.\]
As \eqref{E:est-away-0} provides control on the last term in
\eqref{lowComm2} away from $x=0$ and $x=1$,
it suffices to prove
\[\int_0^T \|\chi k \tilde{\psi}(D_x/k) u\|^2_{L^2}\,dt \leq C_T
\|k^{\beta(m_1,m_2)} u_0\|^2_{L^2}.\]
Here $\chi$ is a cutoff which is $1$ in a neighborhood of the trapped
geodesics at $x=0$ and $x=1$.  The desired bound, thus, will follow
once $u_{hi}$ is controlled as the precise choice of cutoff $\psi$ is inessential.


\subsubsection{High frequency estimate}
It remains to estimate $u_{hi}$ in the vicinity of $x=0$ and $x=1$.
We fix a cutoff $\chi\in \C^\infty_c(\R)$ which is $1$ in a
neighborhood of $x=0$ and in a neighborhood of $x=1$.  Let
\[F(t)g = \chi(x) \psi(D_k/k) k^{r} e^{-itP_k} g,\]
where the constant $r>0$ will be determined later.  We seek to
determine $r$ so that $F:L^2_x \to L^2([0,T];L^2_x)$ as
\begin{equation}\label{E:l-sm-r}\|k^{1-r}
  F(t)u_0\|_{L^2([0,T];L^2_x)}\leq C_T \|k^{1-r} u_0\|_{L^2}\end{equation}
is a local smoothing estimate.  $F$ is such a mapping if and only if
$F F^*:L^2L^2\to L^2L^2$, where we have abbreviated $L^2([0,T];L^2_x) = L^2L^2$.  A straightforward computation shows
that
\[F F^* f(t,x) = \chi(x)\psi(D_x/k) k^{2r} \int_0^T
e^{-i(t-s)P_k} \psi(D_x/k) \chi(x) f(s,x)\,ds,\]
and 
\[\|F F^* f\|_{L^2L^2}\leq C_T \|f\|_{L^2L^2}\]
is the desired estimate.  
We write $F F^* f(t,x) = \chi(x) \psi(D_x/k) (v_1+v_2)$
where
\begin{align*}
  v_1 &= k^{2r} \int_0^t e^{-i(t-s)P_k} \psi(D_x/k) \chi(x)
  f(s,x)\,ds,\\
v_2 &= k^{2r} \int_t^T e^{-i(t-s)P_k}\psi(D_x/k) \chi(x) f(s,x)\,ds.
\end{align*}
Thus,
\[(D_t+P_k) v_{l} = (-1)^{l} i k^{2r} \psi(D_x/k) \chi(x) f,
\quad l=1,2 \]
and
\[\|\chi \psi v_{l}\|_{L^2L^2}\leq C_T \|f\|_{L^2L^2}\]
would imply the desired estimate.  By Plancherel's theorem, this is
equivalent to showing
\[\|\chi \psi \hat{v}_{l}\|_{L^2L^2}\leq C_T \|\hat{f}\|_{L^2L^2}\]
where $\hat{f}$ denotes the Fourier transform of $f$ in the time variable.
I.e., we are required to show that
\[\|\chi \psi k^{2r}(\tau\pm i0 + P_k)^{-1} \psi
\chi\|_{L^2_x\to L^2_x} = O(1)\]
uniformly in $\tau$.  Setting, as above, $-z=\tau k^{-2}$, $h=k^{-1}$, and $V=A^{-2}(x)+h^2 V_1(x)$, we
need
\begin{equation}
\label{E:cutoff-inverse}
\| \chi(x)\psi(hD_x) (-z \pm i 0+ (hD_x)^2 + V)^{-1} \psi(hD_x) \chi(x) 
\|_{L^2 \to L^2} \leq C h^{-2(1-r)}.
\end{equation}

We recall
\[P = (hD_x)^2+V\]
and shall use gluing techniques to reduce proving
\begin{equation}\label{glued}\|\rho_{-s}(P-z)u\|_{L^2} \geq c
  h^{-2\beta(m_1,m_2)} \|\rho_s u\|_{L^2},\quad s<-1/2\end{equation}
which implies \eqref{E:cutoff-inverse} with $1-r=\beta(m_1,m_2)$ as desired,
to proving microlocal invertibility estimates near the trapped sets.  
In
\eqref{glued}, $\rho_s$ is a smooth function that is identity on a
large compact set and is equivalent to $\la x\ra^s$ near infinity.

The gluing techniques that we shall employ are outlined in
\cite{Chr-inf-deg}.  See also \cite[Proposition
2.2]{Chr-disp-1} and \cite{DaVa-gluing}.  

Recall that we are working in $T^*\R$ with principal symbol
$p=\xi^2+V(x)$ where the potential $V(x)$ is a short range
perturbation of $x^{-2}$ and has critical points at precisely
$x=0,1$.  The critical point at $x=0$ is a maximum with value $1$ while the critical
point at $1$ is an inflection point with potential value $C_1^{-1}$.
This means that in terms of the Hamiltonian vector field, $H_p$, the
level set $\{p=1\}$ contains the critical point $(0,0)$ and the level
set $\{p=C_1^{-1}\}$ contains the critical point $(1,0)$.
Furthermore, $\pm V'(x)\le 0$ for $\pm x\ge 0$ with equality only at
these critical points.

As in \cite{Chr-inf-deg}, we fix a few cutoffs.  Let $M>1$ be
sufficiently large so that there is a symbol $p_0$ such that $p_0=p$
for $|x|\ge M-1$ and the operator $P_0$ associated to symbol $p_0$
satisfies 
\[\|\rho_{-s}(P_0-z)u\|_{L^2} \geq c
  \frac{h}{\log(1/h)} \|\rho_s u\|_{L^2}.\]
Such a $P_0$ is, e.g., the $m=1$ case of \cite{ChWu-lsm} and such bounds
follow from \cite{Chr-NC, Chr-QMNC}.
Here $\rho_s$ is a smooth function such that $\rho_s>0$,
$\rho_s(x)\equiv 1$ on a neighborhood of $\{|x|\le 2M\}$, and
$\rho_s\equiv \la x\ra^{s}$ for $x$ sufficiently large.  We choose
$\Gamma\in \C_c^\infty(\R)$ with $\Gamma\equiv 1$ on $\{|x|\le M-1\}$
wth support in $\{|x|\le M\}$.  In particular, $p=p_0$ on supp$(1-\Gamma)$.

Let
\begin{equation}\label{lambda}\Lambda(r):=\int_0^r \la t\ra^{-1-\epsilon_0}\,dt,
\end{equation}
for some fixed $\epsilon_0>0$, which is a function chosen to be
globally bounded with positive derivative and $\Lambda(r)\sim r$ near
$r=0$.  Then let
\[a(x,\xi)=\Lambda(x)\Lambda(\xi),\]
so that
\begin{align*}
  H_p a &= (2\xi\partial_x - V'(x)\partial_\xi)a\\
&= 2\xi\Lambda(\xi)\Lambda'(x) - V'(x)\Lambda(x)\Lambda'(\xi).
\end{align*}
Since $\pm V'(x)<0$ for $\pm x>0$, $x\neq 1$, for any $\epsilon>0$ we
have
\begin{equation}\label{lbhp}
H_p a \ge c_0 > 0,\quad |x|\in [\epsilon/2,1-(\epsilon/2)]\cup [1+(\epsilon/2),M].\end{equation}
We
further have
\begin{equation}\label{lbhp2}H_p a\ge c_0' >0,\quad |\xi|\ge \delta>0 \text{ and } |x|\le M.
\end{equation}

For $j=0,1$, let $\Gamma_j(x)$ be equal to $1$ for $|x-j|\leq
\epsilon/2$ with support in $\{|x-j|\leq \epsilon\}$.  And set
\[\Gamma_2 = \Gamma - \Gamma_0 - \Gamma_1\]
so that $\Gamma_2$ is supported where $|x|\in
[\epsilon/2,1-(\epsilon/2)]\cup[1+(\epsilon/2),M]$.  

We may use a commutator argument to prove the necessary microlocal
black box estimate for $\Gamma_2$.  Indeed, for any $z\in \R$, using \eqref{lbhp},
\begin{align*}
  2\Im\la(P-z)\Gamma_2u,a^w \Gamma_2u\ra &=  -i\la(P-z)\Gamma_2 u, a^w
  \Gamma_2 u\ra +i\la a^w \Gamma_2 u,(P-z)\Gamma_2 u\ra\\
&= i\la [P,a^w]\Gamma_2 u,\Gamma_2 u\ra \geq c_1 h\la
\Gamma_2u,\Gamma_2 u\ra
\end{align*}
for some $c_1>0$.  Then
\[c_1 h \|\Gamma_2 u\|^2 \le 2 \|(P-z)\Gamma_2 u\|\|a^w\Gamma_2 u\|
\le C\|(P-z)\Gamma_2 u\|\|\Gamma_2 u\|,\]
so that
\[\|\Gamma_2 u\|\le C'h^{-1} \|(P-z)\Gamma_2 u\|,\]
as desired.

We now choose two microlocal cutoffs.  For $j=0,1$, let
$\psi_j=\psi_j(\xi^2 + V(x))$ be functions of the principal symbol
$p$.  For some $\delta>0$ to be fixed momentarily, assume
$\psi_0\equiv 1$ for $\{|p-1|\leq \delta\}$ with slightly larger
support and similarly $\psi_1\equiv 1$ for $\{|p-C_1^{-1}|\leq
\delta\}$ with slightly larger support.  The parameter $\delta>0$ may
now be fixed, depending on $\varepsilon>0$, so that $\psi_1\equiv 1$
on $\supp(\Gamma_1)\cap \{\xi\equiv 0\}$.  These cutoffs are depicted
in Figure \ref{fig:cutoffs1}.  Repeating the commutator
argument above but instead using \eqref{lbhp2} allows us to conclude
\[\|\Gamma_0 (1-\psi_0) u\|\le Ch^{-1} \|(P-z)\Gamma_0 (1-\psi_0)
u\|\]
and
\[\|\Gamma_1(1-\psi_1)u\|\le Ch^{-1}\|(P-z)\Gamma_1(1-\psi_1) u\|.\]

\begin{figure}
\hfill
\centerline{\input{cutoffs1}}
\caption{\label{fig:cutoffs1}  The cutoff functions used to apply \cite[Appendix]{Chr-inf-deg}.}
\hfill
\end{figure}
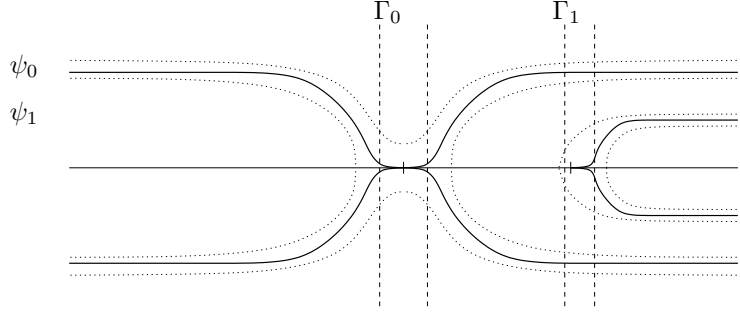

We also include a separate figure (Figure \ref{fig:cutoffs}) that
illustrates that such a microlocalization can be carried out in the
case described in Remark \ref{separates}.

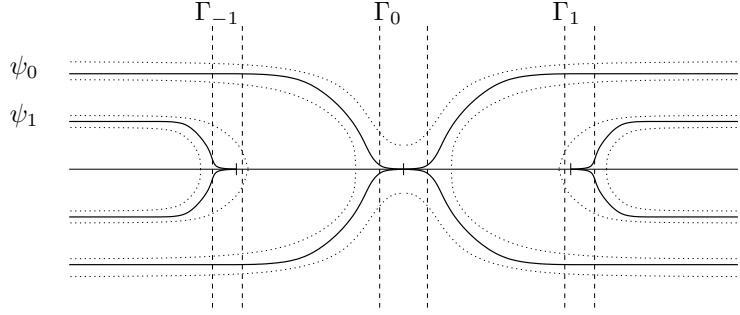
\begin{figure}
\hfill
\centerline{\input{cutoffs}}
\caption{\label{fig:cutoffs}  The cutoff functions in the case
  described in Remark \ref{separates}.}
\hfill
\end{figure}

We may now conclude \eqref{glued} provided that we can establish
such microlocal invertbility estimates for $\Gamma_j \psi_j u$.

The invertibility estimate near $(0,0)$ has been proved
in \cite{Chr-NC,Chr-QMNC} for $m_1 = 1$ and in \cite{ChWu-lsm} for
$m_1 >1$.  For convenience, this is restated
\begin{lemma}
\label{L:m0-inv}
For $\epsilon>0$ sufficiently small, let $\phi \in \s(T^* \reals)$
have compact support in $\{ |(x,\xi) |\leq \epsilon\}$.  Then there
exists $C_\epsilon>0$ such that 
\begin{equation}
\label{E:m0-inv}
\| (P-z) \phi^w u \| \geq C_\epsilon
h^{2m_1/(m_1+1)}  \|
\phi^w u \|, \,\,\, z \in [1-\epsilon, 1 + \epsilon],
\end{equation}
if $m_1>1$.  If $m_1=1$, then $h^{2m_1/(m_1+1)}$ is replaced by $h/(\log(1/h))$.
\end{lemma}

We need only prove the corresponding estimate near $(1,0)$.  
This is also used to prove Theorem \ref{T:resolvent}.
\begin{lemma}
\label{L:ml-inv}
For $\epsilon>0$ sufficiently small, let $\phi \in \s(T^* \reals)$
have compact support in $\{ |(x-1,\xi) |\leq \epsilon\}$.  Then there
exists $C_\epsilon>0$ such that 
\begin{equation}
\label{E:ml-inv}
\| (P-z) \phi^w u \| \geq C_\epsilon
h^{(4m_2+2)/(2m_2+3)}  \|
\phi^w u \|, \,\,\, z \in [C_1^{-1}-\epsilon, C_1^{-1} + \epsilon].
\end{equation}
\end{lemma}

The proof of this estimate proceeds through
several steps.  First, we rescale the principal symbol of $P$ to
introduce a calculus of two parameters.  We then quantize in the
second parameter which eventually will be fixed as a constant in the
problem.  This technique has been used in
\cite{SjZw-mono,SjZw-frac}, \cite{Chr-NC,Chr-QMNC}, \cite{ChWu-lsm}.

\subsection{The two parameter calculus}
Before we proceed to the proof of Lemma \ref{L:ml-inv}, we shall first
review some facts about the two parameter calculus.  These ideas were
introduced in \cite{SjZw-frac}, and we shall employ the
generalizations proved in \cite{ChWu-lsm}.

We set
\begin{eqnarray*}
\lefteqn{\s_{\alpha,\beta}^{k,m, \widetilde{m}} \left(T^*(\R^n) \right):= } \\
& = & \Bigg\{ a \in \Ci \left(\R^n \times (\R^n)^* \times (0,1]^2 \right):  \\
 && \quad \quad  \left| \partial_x^\rho \partial_\xi^\gamma a(x, \xi; h, \tilde{h}) \right| 
\leq C_{\rho \gamma}h^{-m}\tilde{h}^{-\widetilde{m}} \left(
  \frac{\tilde{h}}{h} \right)^{\alpha |\rho| + \beta |\gamma|} 
\langle \xi \rangle^{k - |\gamma|} \Bigg\},
\end{eqnarray*}
when $\alpha \in [0,1]$ and $\beta\le 1-\alpha$.  Throughout, we take
$\th\ge h$.  We abbreviate $\s^{0,0,0}_{\alpha,\beta}$ by
$\s_{\alpha,\beta}$.  The focus shall be on the marginal case
$\alpha+\beta=1$.  In particular, even in this marginal case, we have
that $a\in \s^{k,m,\tilde{m}}_{\alpha,\beta}$ and $b\in
\s^{k',m',\tilde{m}'}_{\alpha,\beta}$ implies that
\[\Op^w_h (a) \circ \Op_h^w (b)=\Op_h^w(c)\]
for some symbol $c\in \s^{k+k',m+m',\tilde{m}+\tilde{m}'}_{\alpha,\beta}$.

We also have the following expansion.  This is from \cite[Lemma
3.6]{SjZw-frac} in the case that $\alpha=\beta=1/2$ and from
\cite{ChWu-lsm} in the more general case.
\begin{lemma}
\label{l:err}
Suppose that 
$ a, b \in \s_{\alpha,\beta}$, 
and that $ c^w = a^w \circ b^w $. 
Then 
\begin{equation}
\label{eq:weylc}  c ( x, \xi) = \sum_{k=0}^N \frac{1}{k!} \left( 
\frac{i h}{2} \sigma ( D_x , D_\xi; D_y , D_\eta) \right)^k a ( x , \xi) 
b( y , \eta) |_{ x = y , \xi = \eta} + e_N ( x, \xi ) \,,
\end{equation}
where for some $ M $
\begin{equation}
\label{eq:new1}
\begin{split}
& | \partial^{\gamma} e_N | \leq C_N h^{N+1}
 \\
& \ \ 
\times \sum_{ \gamma_1 + \gamma_2 = \gamma } 
 \sup_{ 
{{( x, \xi) \in T^* \R^n }
\atop{ ( y , \eta) \in T^* \R^n }}} \sup_{
|\rho | \leq M  \,, \rho \in \NN^{4n} }
\left|
\Gamma_{\alpha, \beta, \rho,\gamma }(D)
( \sigma ( D) ) ^{N+1} a ( x , \xi)  
b ( y, \eta ) 
\right| \,,
\end{split} 
\end{equation}
where $ \sigma ( D) = 
 \sigma ( D_x , D_\xi; D_y, D_\eta )  $ as usual,  
and 
\[
\Gamma_{\alpha, \beta, \rho,\gamma }(D) =( h^\alpha \pa_{(x,y)},
h^\beta \pa_{(\xi,\eta)}))^\rho \partial_{(x,\xi)}^{\gamma_1} 
\partial_{(y,\eta)}^{\gamma_2}.
\]
\end{lemma}

With the scaling of coordinates
\begin{equation}\label{blowdown}
(x,\xi)=\B(X,\Xi)=((h/\th)^\alpha X, (h/\th)^\beta \Xi),
\end{equation}
it follows that if
$a \in \s_{\alpha,\beta}^{k,m,\tilde{m}}$ then
$a \circ \B \in \s_{0,0}^{k,m,\tilde{m}}.
$
Moreover, the unitary operator $T_{h, \tilde{h}} u(X) = \csh^{\frac{n\alpha}{2}}u\left(
  \csh^{\alpha} X \right)$,   
relates the quantizations
\begin{equation}\label{rescaledquantization}
\Op_{\tilde{h}}^w(a\circ \B) T_{h, \tilde{h}} u= T_{h, \tilde{h}} \Op_h^w(a) u.
\end{equation}

\subsection{Proof of Lemma \ref{L:ml-inv}}
Due to the cutoff $\phi^w$, we are working microlocally in
$\{|(x-1,\xi)|\leq \epsilon\}$.  We notice that it suffices to
demonstrate \eqref{E:ml-inv} for $P-z$ replaced by $Q_1=P-h^2V_1-z$ as
$V_1$ is bounded in this region and $(4m_2+2)/(2m_2+3)<2$.
 
Let
\[q_1=\xi^2 + A^{-2}-z\]
be the principal symbol of $Q_1$.  Applying Taylor's theorem about $x=1$ to
$A^{-2}$, we have
\[q_1 = \xi^2 - \frac{c_2}{C_1^2} (x-1)^{2m_2+1}(1+\tilde{a}(x))-z_1\]
where $z_1=z-C_1^{-1}\in [-\epsilon,\epsilon]$ and $\tilde{a}(x)=\O(|x-1|^{2m_2+1})$.
The Hamilton vector field $\hamvf$ associated to the symbol $q_1$ is
\[\hamvf = 2\xi\partial_x +
\Bigl((2m_2+1)\frac{c_2}{C_1^2}(x-1)^{2m_2} + \O(|x-1|^{4m_2+1})\Bigr)\partial_\xi.\]

We introduce the new variables
\[X-1 =
\frac{x-1}{(h/\th)^{\alpha}}, \quad \Xi = \frac{\xi}{(h/\th)^\beta},\]
where
\[\alpha=\frac{2}{2m_2+3},\quad \beta=\frac{2m_2+1}{2m_2+3},\]
and, as above, we shall use $\B$ to denote the map $\B(X-1,\Xi)=(x-1,\xi)$. 
In these new coordinates, we record that
\begin{multline}\label{blownupvf}
\hamvf= (h/\th)^{\frac{2m_2-1}{2m_2+3}}\Bigl(2\Xi \pa_X +
(2m_2+1)\frac{c_2}{C_1^2} 
(X-1)^{2m_2}\pa_\Xi\\+\O((h/\th)^{(2m_2+1)\alpha} |X-1|^{4m_2 +1})\pa_\Xi\Bigr).
\end{multline}

We recall the definition \eqref{lambda} of $\Lambda(r)$, and we
similarly set
\[
\Lambda_2(r)=1+\int_{-\infty}^r \la t\ra^{-1-\epsilon_0}\,dt.\]
Then, for a cutoff function $\chi(s)$ which is identity for
$|s|<\delta_1$ and vanishes for $|s|>2\delta_1$, we introduce
\[a(x,\xi;h,\th) = \Lambda(\Xi)\Lambda_2(X-1)\chi(x-1)\chi(\xi),\]
where $\delta_1>0$ is another parameter which will be fixed shortly.
As $\th\ge h$, we have that
\[
\abs{\pa_X^{\tilde{\alpha}} \pa_\Xi^{\tilde{\beta}} a}\leq C_{\tilde{\alpha},\tilde{\beta}} .
\]
We compute
\[\hamvf(a)=(h/\th)^{\frac{2m_2-1}{2m_2+3}} g(x,\xi;h,\th)+r(x,\xi;h,\th)\]
where
\begin{multline}\label{gdefn}
g = \chi(x-1)\chi(\xi)\Bigl(
2\Lambda(\Xi)\Xi
\ang{X-1}^{-1 - \ep_0} 
\\ +(2m_2+1) \frac{c_2}{C_1^2}  (X-1)^{2m_2}\ang{\Xi}^{-1-\ep_0}
\Lambda_2(X-1) (1+\O(|x-1|^{2m_2+1}))\Bigr)
\end{multline}
and $$\supp r\subset \{\abs{x-1}>\delta_1\} \cup \{\abs{\xi}>\delta_1\}.$$

We first seek to show that the following lemma from \cite{ChWu-lsm}
may be applied to $g$:
\begin{lemma}\label{lemma:positivity0}
Let a real-valued symbol $\tg(x,\xi;h)$ satisfy
$$
\tg(x,\xi;h) =  \begin{cases} c (\xi^2 + x^{2m}) (1 + r_2), & \xi^2 +x^2\leq 1\\
b(x,\xi;h), &  \xi^2 +x^2\geq 1,
\end{cases}
$$
where $c >0$ is constant, $r_2 = \O_{\s_{\alpha, \beta}}( \delta_1),$ and $b>0$ is elliptic.
Then there exists $c_0>0$ such that
$$
\lll\Op_h^w(\tg)u, u \rrr \geq c_0h^{2m/(m+1)} \| u \|_{L^2}^2
$$
for $h$ sufficiently small.
\end{lemma}

In the sequel, we shall only be applying the above to functions which
are microlocally cutoff to the set where $\chi(x-1)\chi(\xi)\equiv 1$.  As the errors off
this set will be $\O(h^\infty)$, we shall assume that $|x-1|\le
\delta_1$ and $|\xi|\le \delta_1$ throughout this discussion.

Over $|(X-1,\Xi)|\leq 1$, we have $\Lambda(\Xi)\sim \Xi$,
$\Lambda_2(X-1)\sim 1$, 
and $\la X-1\ra^{-1-\varepsilon_0}\sim 1$.  Thus, the term $g$, given
in \eqref{gdefn}, of $\hamvf(a)$ is bounded below by a multiple of
$\Xi^2 + (X-1)^{2m_2}$.

We next consider $|(X-1,\Xi)|\geq 1$.  Since
$\sgn\Lambda(s)=\sgn(s)$, when $|\Xi|\geq
\max(\abs{X-1}^{1 + \ep_0}, 1/{4})$, then
$$
g\geq 2\Lambda(\Xi)
\ang{X-1}^{-1-\ep_0}\Xi\gtrsim \frac{\abs{\Xi}}{\ang{\Xi}}\geq C>0.
$$
For $\abs{X-1}^{1 + \ep_0} \geq \max(\abs{\Xi},1/{4}),$ we have
$$
g \geq C' \ang{\Xi}^{-1-\ep_0} \Lambda_2(X-1) (X-1)^{2m_2} \gtrsim
\abs{X-1}^{-(1+\ep_0)^2}\abs{X-1}^{2m_2} \geq C''>0,
$$
provided $(1+\ep_0)^2<2m_2.$  In the region of interest
$|(X-1,\Xi)|\geq 1$, 
the larger of $\abs{\Xi}$ and
$\abs{X-1}^{1 + \ep_0}$ is assuredly greater than $1/{4}$ if
$\epsilon_0>0$ is sufficiently small.  Hence, we
have shown that
$$
g \geq C>0\quad \text{ in } \{\Xi^2+(X-1)^2\geq 1\}.
$$

Recapping, we have found that
$$
\hamvf(a)=(h/\th)^{\frac{2m_2-1}{2m_2+3}}g+r
$$
with 
$$r = \O_{\s_{\alpha, \beta}}((h/\th)^{(2m_2-1)/(2m_2+3)}( (h/\th)^\alpha |
\Xi| + (h/\th)^{\beta } |X-1|^{2m_2}))$$
 supported as above and
$$
g(X,\Xi;h) =  \begin{cases} c (\Xi^2 +  (X-1)^{2m_2}) (1 + r_2) , & \Xi^2 +(X-1)^2\leq 1\\
b, &  \Xi^2 +(X-1)^2\geq 1,
\end{cases}
$$
where $c >0$ is a constant, $r_2 = \O_{\s_{\alpha, \beta}}( \delta_1)$, and $b>0$ is elliptic.

By translating, using the blowdown map $\B$, and relating the
quantizations as in the previous section, we may use Lemma
\ref{lemma:positivity0} to obtain a similar bound on $g$.
\begin{lemma}\label{lemma:positivity1}
For $g$ given by \eqref{gdefn} and  $\tilde{h}>0$ sufficiently small, there exists $c>0$ such that
\[\|\Op_h^w(g\circ \B^{-1})\|_{L^2\to L^2}>c\th^{2m_2/(m_2+1)},\] 
uniformly as $h \downarrow 0$.
\end{lemma}
The proof of this lemma follows exactly as that in \cite{ChWu-lsm} and
is, thus, omitted.


Before completing the proof of Lemma \ref{L:ml-inv}, we need the
following lemma about the lower order terms in the expansion of the
commutator of $Q_1$ and $a^w$.  

\begin{lemma}
\label{L:Q-comm-error}
The symbol expansion of $[Q_1, a^w]$ in the $h$-Weyl calculus is of
the form
\begin{align*}
[Q_1, a^w] = & \Op_h^w \Bigg( \Big( 
\frac{i h}{2} \sigma ( D_x , D_\xi; D_y , D_\eta) \Big) (q_1(x, \xi)
a(y, \eta) - q_1(y, \eta) a ( x , \xi) ) \Bigl|_{ x = y , \xi = \eta} \\
& + e (
x, \xi ) + r_3(x, \xi)\Bigg) ,
\end{align*}
where $r_3$ is supported in $\{|(x,\xi)\geq \delta_1\}$ and $e$ satisfies
\begin{multline*}
\|\Op_h^w(e)\|_{L^2\to L^2} \\\leq C \tilde{h}^{\frac{2m_2 + 7}{2m_2 + 3}-\frac{2m_2}{m_2+1}}
h^{\frac{4m_2 + 2}{2m_2+3}} \Bigl(\|\Op_h^w(g\circ \B^{-1})\|_{L^2\to L^2} + \O(\th^{2+\frac{2m_2}{m_2+1}})\Bigr),
\end{multline*}
with $g$ given by \eqref{gdefn}.
\end{lemma}

\begin{proof}  Since everything is in the Weyl calculus, only the odd
  terms in the exponential composition expansion are non-zero.  In
  accordance with Lemma \ref{l:err},
we set 
\begin{align*}
e& (x,\xi) \\& = \chi(x-1)\chi(\xi) \\
& \qquad \times \sum_{k=1}^{m_2-1}
  \frac{2}{(2k+1)!}\Bigl(\frac{ih}{2}\sigma(D)\Bigr)^{2k+1}
  q_1(x,\xi)\Lambda((\th/h)^\beta \eta)\Lambda_2((\th/h)^\alpha
  (y-1))\Bigl|_{\substack{x=y\\ \xi=\eta}}\\ & \qquad + \chi(\xi)\chi(x-1)e_{2m_2}(x,\xi).
\end{align*}
Here we have extracted the terms in the expansion where derivatives
fall on the cutoff $\chi(\eta)$ of $a$ as these terms have supports
compatible with $r_3$.  For convenience, however, $e_{2m_2}$ denotes
the full error in the expansion of $[Q_1,a^w]$.

Recalling that $q_1(x,\xi)=\xi^2-(x-1)^{2m_2+1}(1+\tilde{a}(x))$, it
follows that
\begin{align*}\tilde{e}_k&:=h^{2k+1}\chi(x-1)\chi(\xi)
\sigma(D)^{2k+1}q_1(x,\xi) \Lambda((\th/h)^\beta
\eta)\Lambda_2((\th/h)^\alpha (y-1)) \Bigl|_{\substack{x=y\\ \xi=\eta}}\\ &= h^{2k+1}\chi(x-1)\chi(\xi) D_x^{2k+1}q_1(x,\xi) D_\eta^{2k+1}
\Lambda((\th/h)^\beta \eta)\Lambda_2((\th/h)^\alpha (y-1)) \Bigl|_{\substack{x=y\\ \xi=\eta}}\\
&=ch^{2k+1}(x-1)^{2m_2+1-(2k+1)}(1+\O((x-1)^{2m_2+1}))\\
& \qquad \qquad \times (\th/h)^{(2k+1)\beta}\Lambda^{(2k+1)}((\th/h)^\beta\xi)\\&\qquad\qquad\times
\Lambda_2((\th/h)^\alpha (x-1))\chi(x-1)\chi(\xi)
\end{align*}
for $1\le k\le m_2-1$.

In order to estimate $e$, we first estimate each $\tilde{e}_k$, $1\le
k\le m_2-1$, using
conjugation to the $2$-parameter calculus.  We have
\[\|\Op^w_h (\tilde{e}_k) u\|_{L^2} = \|T_{h,\th}  \Op^w_h(\tilde{e}_k)
T^{-1}_{h,\th} T_{h,\th} u\|_{L^2}\leq \|T_{h,\th} \Op^w_h(\tilde{e}_k)
T_{h,\th}\|_{L^2\to L^2} \|u\|_{L^2}\]
since $T_{h,\th}$ is unitary.  We recall that
$T_{h,\th}\Op^w_h(\tilde{e}_k)T^{-1}_{h,\th} =
\Op_{\th}^w(\tilde{e}_k\circ \B)$ and note that
\begin{multline*}
\tilde{e}_k\circ\B =c h^{2k+1} (h/\th)^{(2m_2+1-(2k+1))\alpha-(2k+1)\beta}
(X-1)^{2m_2+1-(2k+1)} \\\times  (1+\O((x-1)^{2m_2+1}))
\Lambda^{(2k+1)}(\Xi)
\Lambda_2(X-1)\chi(x-1)\chi(\xi),
\end{multline*}
which can be estimated by
\[C h^{\frac{4m_2+2}{2m_2+3}}
\th^{\frac{2m_2+7}{2m_2+3}}\th^{2(k-1)}(X-1)^{(2m_2+1)-(2k+1)} \Lambda^{(2k+1)}(\Xi)\chi(x-1)\chi(\xi).\]

On $|X-1|\leq 1$, we have that
\[k=(X-1)^{(2m_2+1)-(2k+1)} \Lambda^{(2k+1)}(\Xi)\chi(x-1)\chi(\xi)\]
is bounded, and thus,
\[\|\Op_{\th}^w (k)\|_{L^2\to L^2} \le C \th^{-2m_2/(m_2+1)}
\|\Op_h^w(g\circ \B^{-1})\|_{L^2\to L^2}\]
by Lemma \ref{lemma:positivity1}.  While on $|X-1|\geq 1$, we have
$k\le g$, and thus
\[\|\Op_{\th}^w (k)\|_{L^2\to L^2} \le \|\Op^w_\th (g)\|_{L^2\to L^2}
+ O(\th^2) \le \|\Op^w_h (g\circ \B^{-1})\|_{L^2\to L^2} + O(\th^2).\]

For $e_{2m_2}$, by the standard $L^2$ continuity theorem of
$h$-pseudodifferential operators, it suffices to estimate a finite
number of derivatives of the error $e_{2m_2}$.
We note the bound of Lemma \ref{l:err}
\[
|\partial^\gamma e_{2m_2}|\leq C h^{2m_2+1}
\sum_{\gamma_1+\gamma_2=\gamma} \sup_{\substack{(x,\xi)\in
    T^*\R^n\\(y,\eta)\in T^*\R^n \\ \rho\in \NN^{4n}, |\rho|\le M }}
|\Gamma_{\alpha,\beta,\rho,\gamma}(D)
(\sigma(D))^{2m_2+1} q_1(x,\xi)a(y,\eta)|.
\]
We have
\begin{multline*}(\sigma(D))^{2m_2+1}q_1(x,\xi)a(y,\eta) \\= c (1+\O(x-1)^{2m_2+1})
\chi(y-1)\Lambda((\th/h)^\alpha (y-1))
D^{2m_2+1}_\eta[\Lambda((\th/h)^\beta\eta) \chi(\eta)].\end{multline*}
The last factor is $\O((\th/h)^{(2m_2+1)\beta})$.  Moreover, the
derivatives $h^\beta\partial_\eta$ and $h^\alpha\partial_y$ preserve
the order of $h$ and increase the order of $\th$, while the other
derivatives lead to higher powers of $h/\th\ll 1$.  It, thus, follows
that $|\partial^\gamma (\chi(x-1)\chi(\xi)e_{2m_2})|$ is
\[\O(h^{(4m_2+2)/(2m_2+3)} \th^{(2m_2+1)^2/(2m_2+3)}),\]
and thus, when also combined with Lemma \ref{l:err} satisfies the
given bound.
\end{proof}

We now complete the proof of Lemma \ref{L:ml-inv}.  We set
$v=\varphi^w u$ where $\varphi$ has support where
$\chi(x)\chi(\xi)=1$, and in particular, away from the support of $r_3$.

Then
Lemmas \ref{lemma:positivity1} and \ref{L:Q-comm-error} yield
\begin{align*}
i\ang{[Q_1,a^w]v,v}&=h\ang{\Op_h^w(\hamvf(a))v,v}+\ang{\Op_h^w(e)u,u}
+ \O(h^\infty)\|v\|_{L^2}^2
\\
&= h (h/\th)^{\frac{2m_2-1}{2m_2+3}} \ang{\Op_h^w(g\circ
  \B^{-1})v,v}+\ang{\Op_h^w(e)u,u} + \O(h^\infty)\|v\|_{L^2}^2
\\
&= h^{\frac{4m_2+2}{2m_2+3}}\big( \th^{-\frac{2m_2-1}{2m_2+3}} +\O(\th^{\frac{2m_2 + 7}{2m_2 +
  3}-\frac{2m_2}{m_2+1}})\big)\ang{\Op_h^w(g\circ
\B^{-1})v,v}\\&\qquad\qquad
\qquad\qquad\qquad+ (\O(h^\infty)+\O(\th^{2+\frac{2m_2}{m_2+1}}))\|v\|_{L^2}^2
\\
&\geq C
h^{\frac{4m_2+2}{2m_2+3}} \th^{1 + \frac{4}{2m_2+3} - \frac{2}{m_2 + 1}} \norm{v}_{L^2}^2,
\end{align*}
 for
$\th$ sufficiently small.
The Schwarz inequality and the $L^2$ continuity theorem for
$h$-pseudodifferential operators guarantees 
$$
\big\lvert \ang{[Q_1,a^w]v,v}\big\rvert \leq C \|Q_1 v\|_{L^2} \|v\|_{L^2},
$$
and thus the desired bound with $1\gg \th>0$ fixed. \qed

\section{Quasimodes}

We end by constructing quasimodes near $(1,0)$ in phase space and use
these to saturate the estimate of Proposition \ref{P:smoothing}, and
hence that of Theorem \ref{T:smoothing}.  The proofs follow from
straightforward modifications of those in \cite{ChWu-lsm}.  We, thus,
only provide a terse description.  

Quasimodes were already constructed
near $(0,0)$ in \cite{ChWu-lsm}.  We focus only on the construction
near the inflection point.  We let
\[
\tP = -h^2 \partial_x^2 - c_2 (x-1)^{2m_2 +1}
\]
near $x=1$ and construct quasimodes that are localized to a small
neighborhood of $x=1$.

We set
\[
\gamma = \frac{4m+2}{2m+3},
\]
$E = (\alpha + i \beta)h^{\gamma}$ where $\alpha, \beta>0$ and are
independent of $h$, and 
\[
\varpi(x) = \int_1^{x} (E + c_2 (y-1)^{2m_2+1})^{1/2} dy,
\]
where the branch of the square root is chosen to have positive
imaginary part.  Letting 
\[
u(x) = (\varpi')^{-1/2} e^{i \varpi  / h},
\]
we see that
\[
(hD)^2 u = (\varpi')^2 u + f u,
\]
where
\[
f  = -h^2 \left( \frac{3}{4} (\varpi')^{-2} (\varpi'')^2 - \frac{1}{2}
  (\varpi')^{-1} \varpi ''' \right).
\]

Straightforward modifications of the proof contained in \cite[Lemma
3.1]{ChWu-lsm} yield the following:
\begin{lemma}
The phase function $\varpi$ satisfies the following properties:
\begin{description}

\item[(i)]  There exists $C>0$ independent of $h$ such that 
\[
| \Im \varpi | \leq C
 h.
\]
In particular, if $| x -1| \leq C h^{\gamma/(2m_2+1)}$, $| \Im \varpi| \leq C'$
for some $C'>0$ independent of $h$.

\item[(ii)]  There exists $C>0$ independent of $h$ such that if
  $\delta>0$ is sufficiently small and $| x -1| \leq \delta
  h^{\gamma/(2m_2+1)}$, then 
\[
C^{-1} h^{\gamma/2} \leq | \varpi'(x) | \leq C 
  h^{\gamma/2}.
\]

\item[(iii)]
\begin{align*}\varpi' &= (E + c_2 (x-1)^{2m_2+1})^{1/2}, \\
\varpi'' &= \frac{1}{2} c_2 (2m_2+1)(x-1)^{2m_2} (\varpi')^{-1}, \\
\varpi''' &= \Bigl( \frac{1}{2}c_2(2m_2+1) (2m_2)(E (x-1)^{2m_2-1} +
  c_2(x-1)^{4m_2})\\&\qquad\qquad\qquad\qquad\qquad\qquad - \frac{1}{4} c_2^2 (2m_2+1)^2 (x-1)^{4m_2}
\Bigr) ( \varpi')^{-3}.
\end{align*}
In particular, there are constants $C_{m_2,1}, C_{m_2,2}$ such that 
\[
f = -h^2 \left( C_{m_2,1} (x-1)^{4m_2} +  C_{m_2,2}E  (x-1)^{2m_2-1}\right) (\varpi'
)^{-4}.
\]

\end{description}

\end{lemma}

From this, we obtain that $|u(x)|\sim |\varphi'|^{-1/2}$ for all $x$.
We localize $u$ by setting 
\[\mu=\delta h^{\gamma/(2m_2+1)},\quad 0<\delta\ll 1\]
fixing $\chi(s)\in \C^\infty_c(\R)$ so that $\chi\equiv 1$ for
$|s|\leq 1$ and $\supp\chi\subset [-2,2]$, and letting
\[
\tu(x) = \chi((x-1)/\mu) u(x).
\]
More calculations, which are again in the spirit of those contained in
\cite{ChWu-lsm}, show that $\|\tu\|^2_{L^2}\sim
h^{(1-2m_2)/(2m_2+3)}$ and
\[(hD)^2 \tu = (\varpi')^2 \tu + R,
\]
where 
\[
R = f \tu + [(hD)^2, \chi((x-1)/\mu)] u.
\]
Moreover, the remainder satisfies
\begin{equation}
\label{E:R-remainder}
\| R \|_{L^2} = \O (h^{\gamma}) \| \tu \|_{L^2}.
\end{equation}




This quasimode can then be used to saturate the local smoothing
estimates near the inflection point.  We, again, refer the interested
reader to the proof in \cite[Theorem 3]{ChWu-lsm}, which provides the
following.

\begin{theorem}
\label{T:sharp}

Let $\phi_0(x, \theta) = e^{ik \theta} \tu(x)$, where $\tu \in \Ci_c (
\reals)$ was constructed above.  We let $h=\abs{k}^{-1},$ where $|k|$ is
taken sufficiently large.  Suppose $\psi$
solves
\[
\begin{cases}
(D_t - \tDelta) \psi = 0, \\
\psi|_{t=0} = \phi_0.
\end{cases}
\]
Then for any $\chi
\in \Ci_c( \reals)$ such that $\chi \equiv 1$ on $\supp \tu$ and
$A>0$ sufficiently large, independent of $k$, there exists a constant
$C_0>0$ independent of $k$ such that 
\begin{equation}\label{saturated}
\int_0^{|k|^{-4/(2m_2+3)}/A} \| \lll D_\theta \rrr \chi \psi \|_{L^2}^2
dt \geq C_0^{-1} \| \lll D_\theta \rrr^{(2m_2+1)/(2m_2+3)} \phi_0 \|_{L^2}^2.
\end{equation}

\end{theorem}

\bibliographystyle{alpha}
\bibliography{trans-sm-bib}

\end{document}

%% file: mfld.tex
\begin{picture}(0,0)%
\includegraphics{mfld.pstex}%
\end{picture}%
\setlength{\unitlength}{1973sp}%
\begingroup\makeatletter\ifx\SetFigFont\undefined%
\gdef\SetFigFont#1#2#3#4#5{%
  \reset@font\fontsize{#1}{#2pt}%
  \fontfamily{#3}\fontseries{#4}\fontshape{#5}%
  \selectfont}%
\fi\endgroup%
\begin{picture}(9774,6422)(289,-7361)
\put(7876,-7261){\makebox(0,0)[lb]{\smash{{\SetFigFont{10}{12.0}{\rmdefault}{\mddefault}{\updefault}{\color[rgb]{0,0,0}$x = 1$}%
}}}}
\put(4651,-7261){\makebox(0,0)[lb]{\smash{{\SetFigFont{10}{12.0}{\rmdefault}{\mddefault}{\updefault}{\color[rgb]{0,0,0}$x = 0$}%
}}}}
\end{picture}%

%% file: AB.tex
\begin{picture}(0,0)%
\includegraphics{AB.pstex}%
\end{picture}%
\setlength{\unitlength}{1973sp}%
\begingroup\makeatletter\ifx\SetFigFont\undefined%
\gdef\SetFigFont#1#2#3#4#5{%
  \reset@font\fontsize{#1}{#2pt}%
  \fontfamily{#3}\fontseries{#4}\fontshape{#5}%
  \selectfont}%
\fi\endgroup%
\begin{picture}(7254,2842)(2359,-3161)
\put(9301,-3061){\makebox(0,0)[lb]{\smash{{\SetFigFont{10}{12.0}{\rmdefault}{\mddefault}{\updefault}{\color[rgb]{0,0,0}$r$}%
}}}}
\put(8626,-2311){\makebox(0,0)[lb]{\smash{{\SetFigFont{10}{12.0}{\rmdefault}{\mddefault}{\updefault}{\color[rgb]{0,0,0}$B(r)$}%
}}}}
\put(4501,-586){\makebox(0,0)[lb]{\smash{{\SetFigFont{10}{12.0}{\rmdefault}{\mddefault}{\updefault}{\color[rgb]{0,0,0}$A(r)$}%
}}}}
\put(4051,-3061){\makebox(0,0)[lb]{\smash{{\SetFigFont{10}{12.0}{\rmdefault}{\mddefault}{\updefault}{\color[rgb]{0,0,0}$3$}%
}}}}
\put(3451,-3061){\makebox(0,0)[lb]{\smash{{\SetFigFont{10}{12.0}{\rmdefault}{\mddefault}{\updefault}{\color[rgb]{0,0,0}$2$}%
}}}}
\put(2851,-3061){\makebox(0,0)[lb]{\smash{{\SetFigFont{10}{12.0}{\rmdefault}{\mddefault}{\updefault}{\color[rgb]{0,0,0}$1$}%
}}}}
\put(5101,-3061){\makebox(0,0)[lb]{\smash{{\SetFigFont{10}{12.0}{\rmdefault}{\mddefault}{\updefault}{\color[rgb]{0,0,0}$r$}%
}}}}
\put(8251,-3061){\makebox(0,0)[lb]{\smash{{\SetFigFont{10}{12.0}{\rmdefault}{\mddefault}{\updefault}{\color[rgb]{0,0,0}$3$}%
}}}}
\put(7651,-3061){\makebox(0,0)[lb]{\smash{{\SetFigFont{10}{12.0}{\rmdefault}{\mddefault}{\updefault}{\color[rgb]{0,0,0}$2$}%
}}}}
\put(7051,-3061){\makebox(0,0)[lb]{\smash{{\SetFigFont{10}{12.0}{\rmdefault}{\mddefault}{\updefault}{\color[rgb]{0,0,0}$1$}%
}}}}
\end{picture}%

%% file: mfld-trans.tex
\begin{picture}(0,0)%
\includegraphics{mfld-trans.pstex}%
\end{picture}%
\setlength{\unitlength}{3947sp}%
\begingroup\makeatletter\ifx\SetFigFont\undefined%
\gdef\SetFigFont#1#2#3#4#5{%
  \reset@font\fontsize{#1}{#2pt}%
  \fontfamily{#3}\fontseries{#4}\fontshape{#5}%
  \selectfont}%
\fi\endgroup%
\begin{picture}(3787,3834)(2979,-3973)
\put(6676,-2311){\makebox(0,0)[lb]{\smash{{\SetFigFont{10}{12.0}{\rmdefault}{\mddefault}{\updefault}{\color[rgb]{0,0,0}$x_1$}%
}}}}
\put(4651,-286){\makebox(0,0)[lb]{\smash{{\SetFigFont{10}{12.0}{\rmdefault}{\mddefault}{\updefault}{\color[rgb]{0,0,0}$x_3$}%
}}}}
\put(6751,-1861){\makebox(0,0)[lb]{\smash{{\SetFigFont{10}{12.0}{\rmdefault}{\mddefault}{\updefault}{\color[rgb]{0,0,0}$C$}%
}}}}
\put(6376,-811){\makebox(0,0)[lb]{\smash{{\SetFigFont{10}{12.0}{\rmdefault}{\mddefault}{\updefault}{\color[rgb]{0,0,0}$x_2$}%
}}}}
\end{picture}%

%% file: cutoffs1.tex
\begin{picture}(0,0)%
\includegraphics{cutoffs1.pstex}%
\end{picture}%
\setlength{\unitlength}{1973sp}%
\begingroup\makeatletter\ifx\SetFigFont\undefined%
\gdef\SetFigFont#1#2#3#4#5{%
  \reset@font\fontsize{#1}{#2pt}%
  \fontfamily{#3}\fontseries{#4}\fontshape{#5}%
  \selectfont}%
\fi\endgroup%
\begin{picture}(9187,3954)(436,-5173)
\put(5026,-1486){\makebox(0,0)[lb]{\smash{{\SetFigFont{10}{12.0}{\rmdefault}{\mddefault}{\updefault}{\color[rgb]{0,0,0}$\Gamma_0$}%
}}}}
\put(7276,-1486){\makebox(0,0)[lb]{\smash{{\SetFigFont{10}{12.0}{\rmdefault}{\mddefault}{\updefault}{\color[rgb]{0,0,0}$\Gamma_1$}%
}}}}
\put(451,-2161){\makebox(0,0)[lb]{\smash{{\SetFigFont{10}{12.0}{\rmdefault}{\mddefault}{\updefault}{\color[rgb]{0,0,0}$\psi_0$}%
}}}}
\put(451,-2761){\makebox(0,0)[lb]{\smash{{\SetFigFont{10}{12.0}{\rmdefault}{\mddefault}{\updefault}{\color[rgb]{0,0,0}$\psi_1$}%
}}}}
\end{picture}%

%% file: cutoffs.tex
\begin{picture}(0,0)%
\includegraphics{cutoffs.pstex}%
\end{picture}%
\setlength{\unitlength}{1973sp}%
\begingroup\makeatletter\ifx\SetFigFont\undefined%
\gdef\SetFigFont#1#2#3#4#5{%
  \reset@font\fontsize{#1}{#2pt}%
  \fontfamily{#3}\fontseries{#4}\fontshape{#5}%
  \selectfont}%
\fi\endgroup%
\begin{picture}(9187,3954)(436,-5173)
\put(5026,-1486){\makebox(0,0)[lb]{\smash{{\SetFigFont{10}{12.0}{\rmdefault}{\mddefault}{\updefault}{\color[rgb]{0,0,0}$\Gamma_0$}%
}}}}
\put(7276,-1486){\makebox(0,0)[lb]{\smash{{\SetFigFont{10}{12.0}{\rmdefault}{\mddefault}{\updefault}{\color[rgb]{0,0,0}$\Gamma_1$}%
}}}}
\put(451,-2161){\makebox(0,0)[lb]{\smash{{\SetFigFont{10}{12.0}{\rmdefault}{\mddefault}{\updefault}{\color[rgb]{0,0,0}$\psi_0$}%
}}}}
\put(451,-2761){\makebox(0,0)[lb]{\smash{{\SetFigFont{10}{12.0}{\rmdefault}{\mddefault}{\updefault}{\color[rgb]{0,0,0}$\psi_1$}%
}}}}
\put(2776,-1486){\makebox(0,0)[lb]{\smash{{\SetFigFont{10}{12.0}{\rmdefault}{\mddefault}{\updefault}{\color[rgb]{0,0,0}$\Gamma_{-1}$}%
}}}}
\end{picture}%